\documentclass[reqno,12pt,centertags]{amsart}
\usepackage{geometry}
\usepackage{amsmath,amsthm,amscd,amssymb,latexsym,upref,stmaryrd}
\usepackage{graphicx}
\usepackage{hyperref}
\usepackage{enumitem}
\usepackage{xcolor}
\newcommand*{\mailto}[1]{\href{mailto:#1}{\nolinkurl{#1}}}
\usepackage[numbers,sort&compress]{natbib}

\def\th#1{\vspace{1mm}\noindent{\bf #1}\quad}

\newcommand{\beq}{\begin{equation}}
	\newcommand{\eeq}{\end{equation}}
\newcommand{\ba}{\begin{align}}
	\newcommand{\ea}{\end{align}}

 \allowdisplaybreaks
\numberwithin{equation}{section}
\newtheorem{theorem}{Theorem}[section]
\newtheorem{lemma}[theorem]{Lemma}

\theoremstyle{definition}
\newtheorem{definition}[theorem]{Definition}
\newtheorem{proposition}[theorem]{Proposition}

\newtheorem{remark}{Remark}[section]

\begin{document}

	\title[Fourth-order inverse problems]
	{Borg-type theorem for a class of fourth-order differential operators}
	
	\author[A.~W.~Guan]{AI-WEI GUAN}
	\address{Department of Mathematics, School of Mathematics and Statistics, Nanjing University of
		Science and Technology, Nanjing, 210094, Jiangsu, People's
		Republic of China}
	\email{\mailto{guan.ivy@njust.edu.cn}}
	
	\author[C.~F.~Yang*]{CHUAN-FU Yang*}
	\address{Department of Mathematics, School of Mathematics and Statistics, Nanjing University of
		Science and Technology, Nanjing, 210094, Jiangsu, People's
		Republic of China}
	\email{\mailto{chuanfuyang@njust.edu.cn}}

    \author[N.~P.~Bondarenko]{NATALIA P. BONDARENKO}
    \address{S.M. Nikolskii Mathematical Institute, Peoples' Friendship University of Russia (RUDN University), 6 Miklukho-Maklaya Street, Moscow, 117198, Russian Federation.}
    \address{ Moscow Center of Fundamental and Applied Mathematics, Lomonosov Moscow State University, Moscow 119991, Russian Federation.}
    \email{\mailto{bondarenkonp@sgu.ru}}

	\date{\today}
	\footnote{* Corresponding author} 
	
	\begin{abstract}
		{In this paper, we study an inverse spectral problem for the fourth-order differential equation $y^{(4)} - (p y')' + q y = \lambda y$ with real-valued coefficients $p$ and $q$ of $L^2(0,1)$. 
		We prove that, for near constant coefficients, the two spectra 
        corresponding to the Dirichlet and the Dirichlet-Neumann boundary conditions uniquely determine either $p$ or $q$. The result extends the Borg theorem of the second-order case to the fourth-order case.}
	\end{abstract}
	
	\maketitle
	
\section{Introduction}
	This paper deals with inverse spectral problems for the fourth-order operators generated by the differential equation
\begin{equation}\label{1}
	y^{(4)}(x) - (p(x) y(x)')' + q(x) y(x) = \lambda y(x), \quad x \in [0,1],
\end{equation}
with the following boundary conditions:
\begin{align}\label{boud1}
	\text{Dirichlet}: \quad & y(0)=y''(0)=y(1)=y''(1)=0, \\ \label{boud2}
	\text{Dirichlet-Neumann}: \quad	& y(0)=y''(0)=y'(1)=y'''(1)=0,
\end{align}
where $p$ and $q$ belong to the class $L_{\mathbb{R}}^{2}(0,1)$ of real-valued, square-integrable, measurable functions on $[0,1]$, and $\lambda$ is the spectral parameter.

Spectral theory of fourth-order differential operators causes interest of scholars because of applications in mechanics, optics, acoustics, and other fields of science (see \cite{Mik, Pol}). In particular, eigenvalue asymptotics for the boundary value problems \eqref{1}, \eqref{boud1} and \eqref{1}, \eqref{boud2}
have been studied by Schueller \cite{Sc}, Badanin and Korotyaev \cite{Bad}, Polyakov \cite{Pol2}, and other authors. 
In this paper, we focus on inverse spectral problems that consist in the recovery of the coefficients $p$ and $q$ from spectra.

The most complete results in the inverse problem
theory were obtained for the second-order operators (see, e.g., the monographs by Levitan \cite{Lev1}, Marchenko \cite{Mar}, Freiling and Yurko \cite{Fre}, Kravchenko \cite{Kra}, and references therein). In particular, Borg \cite{Bor} demonstrated that the coefficient $q(x)$ of the Sturm-Liouville equation
$$
-y''(x) + q(x) y(x) = \lambda y(x), \quad x \in (0, 1),
$$
is uniquely specified by two spectra corresponding to different sets of boundary conditions, say, $y(0) = y(1) = 0$ and $y(0) = y'(1) = 0$. A constructive method for solving the inverse Sturm-Liouville problem was proposed by Gelfand and Levitan \cite{GL51}. However, the Gelfand-Levitan method is ineffective for higher-order differential equations.

The general theory of inverse spectral problems for the higher-order differential operators 
\begin{equation}\label{high}
y^{(n)}+\sum_{k=0}^{n-2}p_k(x)y^{(k)}, \quad n\geq 2,
\end{equation}
has been developed by Yurko \cite{Yur,Yur1,Yur2} using the method of spectral mappings. Relying on those ideas, Bondarenko \cite{Bon1} obtained necessary and sufficient conditions for solvability of inverse spectral problems for  higher-order operators with distribution coefficients. 

Various issues of inverse spectral theory for the fourth-order differential operators were investigated in \cite{Bad, Bar1, Bon, Guan, Pap, SC1} and other studies. In particular, the generalization of Borg's uniqueness result \cite{Bor}
to equation \eqref{1} was considered by Barcilon \cite{Bar1}, who claimed that the coefficients $p$ and $q$ are uniquely specified by three spectra. However, the proof of uniqueness given in \cite{Bar1} was found to be incorrect \cite{Bon}. 
A correct proof was later provided by Guan et al \cite{Guan} under the so-called separation condition on the spectra.

In this paper, we suppose that one of the coefficients ($p$ or $q$) is known a priori and investigate the uniqueness of recovering the remaining coefficient from a less amount of data. For piecewise analytic coefficients, this issue has been studied by Yurko using the method of standard models (see \cite[Section~4]{Yur3}).
Specifically, for the operator \eqref{high} of even order $n = 2m$, Yurko has shown that, given $(2m-2)$ coefficients, the remaining coefficient is uniquely determined by the two spectra generated by the boundary conditions
\begin{equation*}
y(0)=y'(0)=\dots=y^{(2m-2)}(0)=y^{(j)}(\pi)=0, \quad j = 0, 1.
\end{equation*}
However, for $L^2$ coefficients, this question is still open.
Furthermore, for the even order differential operator \eqref{high} with $p_k =0$ for $k=\overline{1, n-2}$, Khachatryan  \cite{Kh} has demonstrated that, if $p_0(x)$ is a symmetric function about $\frac{1}{2}$ with sufficiently small norm, then the single Dirichlet spectrum uniquely determines $p_0(x)$ in $L^{2}(0,1)$. 
For the fourth-order equation \eqref{1}, Schuller \cite{SC1} has shown that, if $p$ and $q$ are near constant and symmetric functions, then one of them together with a spectrum uniquely specify the other one. In this paper, we remove the symmetricity assumption and prove the uniqueness by two spectra.

Let us proceed to formulating our main result.
Denote by $L_1(p,q)$ and $L_2(p,q)$ the operators generated by equation \eqref{1} with the Dirichlet \eqref{boud1} and the Dirichlet-Neumann \eqref{boud2} boundary conditions, respectively. These operators are understood in the sense of sesquilinear forms (see Subsection~\ref{sec:sesq} for details).
Denote $\mathbf{p}:=(p,q)$, $\mathbf{p_i}:=(p_i,q_i)$, $p_i, q_i \in L_{\mathbb R}^2(0,1)$, $i=\overline{1,4}$, $\mathbf{a}:=(a_1,a_2)$, $a_1,a_2\in \mathbb{R}$, $\mathbf{0}:=(0,0)$, $\mathbf{E}:=L_{\mathbb R}^2(0,1) \times L_{\mathbb R}^2(0,1)$. Let $\{{\lambda}_n (\mathbf{p})\}_{n \ge 1}$ and $\{{\mu}_n (\mathbf{p})\}_{n\geq1}$ be the spectra of the operators $L_1(\mathbf{p})$ and $L_2(\mathbf{p})$, respectively.

Denote by $B(\mathbf{a},\varepsilon)$ the ball of radius $\varepsilon>0$ about $\mathbf{a}$ in $\mathbf{E}$.
Let 
$$
\mathbf{W}=\{\mathbf{p}\in \mathbf{E} \colon L_1(\mathbf{p}) \text{ and } L_2(\mathbf{p}) \text{ have only simple eigenvalues}\}.
$$

Our uniqueness theorem is formulated as follows.
 
\begin{theorem}\label{thm}
Let $\mathbf{a} \in \mathbf{W}$ be constant. Then, there exists an $\varepsilon>0$ such that, if $\mathbf{p_1}$, $\mathbf{p_2}\in B(\mathbf{a},\varepsilon)$, $\lambda_n(\mathbf{p}_1)=\lambda_n(\mathbf{p}_2)$ and $\mu_n(\mathbf{p}_1)={\mu}_n(\mathbf{p}_2)$ for all $n\geq 1$, then

(1) $q_1(x)=q_2(x)$ a.e. on $(0,1)$ implies $p_1(x)=p_2(x)$ a.e. on $(0,1)$;

(2) $p_1(x)=p_2(x)$ a.e. on $(0,1)$ implies $q_1(x)=q_2(x)$ a.e. on $(0,1)$.
\end{theorem}

The proof of Theorem~\ref{thm} is based on developing the technique of \cite{Bor, Cau, SC1}. We compose a special Riesz basis by using the eigenfunctions of the four operators $L_k(\mathbf{p}_i)$, $i,k = 1, 2$. For this purpose, we 
study the asymptotic properties of the eigenfunctions and obtain the uniform estimates for them as $\mathbf{p_i}$ are bounded.

The paper is organized as follows. In Section~\ref{sec:prelim}, we define the fourth-order operators by using sesquilinear forms, explicitly present eigenvalues and eigenfunctions for the unperturbed operators, and recall basic facts about Riesz bases.
In Section~\ref{sec:est}, the estimates for the eigenfunctions of the perturbed operators are derived. Section~\ref{sec:proof} contains the proof of Theorem~\ref{thm}.

\section{Preliminaries} \label{sec:prelim}

\subsection{Sesquilinear forms} \label{sec:sesq}

For non-smooth coefficients $p$ and $q$ in equation \eqref{1}, it is convenient to define the corresponding fourth-order differential operators by the method of sesquilinear forms (see \cite{Sc}, Sect. 5.3). 
Consider the closed sesquilinear forms
\begin{equation}\label{frakq}
\mathfrak{q}_{(p,q)}[y,z]=(y'',z'')+(py',z')+(qy,z),
\end{equation}
on the respective domains
\begin{equation}\label{Q_1}
\mathcal{Q}_1=\{y\in H^2[0,1]: y(0)=y(1)=0\}.
\end{equation}
and
\begin{equation}\label{Q_2}
\mathcal{Q}_2=\{y\in H^2[0,1]: y(0)=y'(1)=0\}.
\end{equation}

According to the form representation theorem (see, for example, Theorem $\text{VI.2.1}$ of \cite{Ka}), these forms are associated with the unique sectorial operators, denote them by $L_1(p, q)$ and $L_2(p, q)$, respectively.
Furthermore, the domains of $L_1(p, q)$ and $L_2(p, q)$ are contained in $H^3[0,1]$, with Dirichlet boundary conditions for $L_1$ and Dirichlet-Neumann conditions for $L_2$. Employing the form representation theorem, it can be demonstrated that, if $p,q\in C^{1}[0,1]$, then $L_1(p,q)$ and $L_2(p,q)$ are generated by the boundary value problems \eqref{1}, \eqref{boud1} and \eqref{1}, \eqref{boud2}, respectively. This fact motivates the choice of
$\mathcal{Q}_1$ and $\mathcal{Q}_2$.

\subsection{Eigenfunctions and eigenvalues of unperturbed operators}

For $\mathbf{p} \in \mathbf{E}$ and $n \ge 1$, denote by $\phi_n(x; \mathbf{p})$ and $\psi_n(x; \mathbf{p})$ the normalized eigenfunctions corresponding to the eigenvalues $\lambda_n(\mathbf{p})$ and $\mu_n(\mathbf{p})$ of the operators $L_1(\mathbf{p})$ and $L_2(\mathbf{p})$, respectively.

For constant coefficients $\mathbf{a} = (a_1, a_2) \in \mathbb R^2$, the eigenfunctions and the eigenvalues can be presented in the explicit form:
\begin{equation*}
	{\phi}_{n}(x;\mathbf{\mathbf{a}})=\sqrt{2} \sin(n\pi x), 
\end{equation*}
\begin{equation} \label{psia}	        
{\psi}_{n}(x;\mathbf{\mathbf{a}})=\sqrt{2} \sin\bigg(\bigg(n+\frac{1}{2}\bigg)\pi x\bigg)
\end{equation}
and
\begin{equation*}
	\lambda_n(\mathbf{a})=n^4\pi^4+a_1n^2\pi^2+a_2,
\end{equation*}
\begin{equation} \label{muna}	       
	\mu_n(\mathbf{a})=\bigg(n+\frac{1}{2}\bigg)^4\pi^4+a_1\bigg(n+\frac{1}{2}\bigg)^2\pi^2+a_2.
\end{equation}

Note that the equalities $\lambda_k(\mathbf{a})=\lambda_{k+m}(\mathbf{a})$ and $\mu_i(\mathbf{b})=\mu_{i+j}(\mathbf{b})$ hold for constant coefficients of the form
\begin{equation*}
	\mathbf{a}_{k,m}=(-\pi^2(k^2+(k+m)^2),a_2),\;\;\;\;k=1,2,\dots, \;\;\;\;m=1,2,\dots,
\end{equation*}
and
\begin{equation*}
	\mathbf{b}_{i,j}=(-\pi^2\bigg(\bigg(i+j+\frac{1}{2}\bigg)^2+\bigg(i+\frac{1}{2}\bigg)^2),a_2\bigg),\;\;\;\;i=1,2,\dots, \;\;\;\;j=1,2,\dots,
\end{equation*}
respectively.

Denote 
$$
\mathbf{A}^c :=  \bigcup\limits_{k,t=1}^\infty\bigg\{\bigg(-\pi^2\bigg(\bigg(\frac{k+1}{2}\bigg)^2+\bigg(\frac{k+1}{2}+t\bigg)^2\bigg),y\bigg): y\in\mathbb{R}\bigg\}.
$$

Thus, there is the non-empty subset $\mathbb R^2 \setminus \mathbf{A}^c$ of constants $\mathbf{a}$ that belong to $\mathbf{W}$.

\subsection{Riesz basis}
\begin{definition}
	A set of vectors $\{e_n\}_{n\geq0}$ is called a Riesz basis of a separable Hilbert space $\mathcal{H}$ if 
	
\begin{itemize}\label{Rie}
	\item $\{ e_n \}_{n \ge 0}$ is complete in $\mathcal{H}$, i.e., if $(u,e_n)=0$ for all $n$ and some $u \in\mathcal{H} $, then $u=0$;
	\item the vectors $\{ e_n \}_{n \ge 0}$ are independent, i.e., if $\sum_{n=0}^{\infty}
	c_ne_n=0$ for $\{c_n\}\in l^2$, then $c_n=0$ for all $n$;
	\item there exist strictly positive constants $a$ and $A $, such that
	the following inequality holds for every $u \in \mathcal H$: $$a\|u\|^2\leq\sum_{n=0}^{\infty}|(u,e_n)|^2\leq A\|u\|^2.$$
\end{itemize}
\end{definition}

An important fact is that a small perturbation of a Riesz basis is also a Riesz basis: 

\begin{proposition}\label{le3}\cite{SC1}
Let $\left\{e_n\right\}_{n \ge 0}$ be a Riesz basis of a separable Hilbert space $\mathcal{H}$. Let $\{ d_n \}_{n \ge 0}$ be a collection of vectors in $\mathcal{H}$ such that 
\begin{equation}\label{*}
\sum_{n=0}^{\infty}\left\|d_n-e_n\right\|^2<\frac{a^2}{A} .
\end{equation}
Then, the vector set $\left\{d_n\right\}_{n \ge 0}$ forms a Riesz basis of $\mathcal{H}$. 	
\end{proposition}

\begin{remark}
In particular, if  $\{e_n\}_{n\geq0}$ is an orthonormal basis of $\mathcal{H}$ (which is also a Riesz basis of $\mathcal{H}$), then 
$$
\|u\|^2=\sum_{n=0}^{\infty}|(u,e_n)|^2. 
$$
Consequently, the condition \eqref{*} in Proposition~\ref{le3} becomes 
$$
\sum_{n=0}^{\infty}\left\|d_n-e_n\right\|^2<1.
$$
\end{remark}

More information about Riesz bases can be found in \cite{Christ}.

\section{Estimates for eigenfunctions} \label{sec:est}

In this section, we obtain technical estimates for the eigenfunctions $\{ \phi_n(x;\mathbf{p})\}$, $\{ \psi_n(x; \mathbf{p})\}$ and their derivatives, which are crucial for the proof of Theorem~\ref{thm}. First of all, we need the analytic dependence of the eigenfunctions on the coefficients $\mathbf{p}$ of equation \eqref{1}:

\begin{proposition}\label{le2}\cite{Cau,SC1}
	The functions $\phi_n(x;\cdot)$, $\phi_n'(x;\cdot)$, $\psi_n(x;\cdot)$ and $\psi_n'(x;\cdot)$ are real-analytic at every point of  $\mathbf{W}$.
\end{proposition}

Introduce the notations
\begin{equation*}
	\tilde{\phi}^{\prime}_n(x;\mathbf{p})=\frac{\phi_n^{\prime}(x;\mathbf{p})}{n\pi}, \quad
	\tilde{\psi}^{\prime}_n(x;\mathbf{p})=\frac{\psi_n^{\prime}(x;\mathbf{p})}{\big(n+\frac{1}{2}\big)\pi}. 
\end{equation*}

The main result of this section is the following lemma.

\begin{lemma}\label{prop}
Let $2<\alpha<3$ and $M>0$. Then, there exists a positive integer $P$, depending only on $\alpha$ and $M$, such that, for any $\mathbf{p}_1, \mathbf{p}_2, \mathbf{p}_3, \mathbf{p}_4 \in B(\mathbf{0}, M)$ and $n>P$, there hold
\begin{align}\label{a1}
&    \sup _{x \in[0,1]}\!\!\!\left| \phi_n(x;\! \mathbf{p}_1) \phi_n(x; \!\mathbf{p}_2)\!\!-\!\!\phi_n\!\left(x ;\! \mathbf{p}_3\right) \phi_n\!\left(x ;\! \mathbf{p}_4\right)\right|\!\!\leq\!\! C\!\left(\|\mathbf{\bf{p}_1}\!\!-\!\!\mathbf{\bf{p}_2}\|\!\!+\!\!\|\mathbf{p_1}\!\!-\!\!\mathbf{p_3}\|\!\!+\!\!\|\mathbf{p_1}\!\!-\!\!\mathbf{p_4}\|
		\right) \!n^{2-\alpha}, \\
		\label{a2}
&		\sup _{x \in[0,1]}\!\!\!\left| \psi_n(x;\!\mathbf{p}_1\!) \psi_n(x;\! \mathbf{p}_2)\!\!-\!\!\psi_n\!\left(\!x ;\!\mathbf{p}_3\right) \psi_n\!\left(\!x ;\mathbf{p}_4\!\right)\right|\!\!\leq\!\!C\!\left(\|\mathbf{\bf{p}_1}\!\!-\!\!\mathbf{\bf{p}_2}\|\!\!+\!\!\|\mathbf{p_1}\!\!-\!\!\mathbf{p_3}\|\!\!+\!\!\|\mathbf{p_1}\!\!-\!\!\mathbf{p_4}\|
		\right) \!n^{2-\alpha}, \\ \label{b1}
&		\sup _{x \in[0,1]}\!\!\!|\tilde{\phi}_n^{\prime}(x;\!\mathbf{p}_1\!) \tilde{\phi}_n^{\prime}(x; \mathbf{p}_2)\!\!-\!\!\tilde{\phi}_n^{\prime}\!\left(x ;\!\mathbf{p}_3\right) \tilde{\phi}_n^{\prime}\!(x ;\!\mathbf{p}_4)|\!\!\leq\!\!C\!\left(\|\mathbf{\bf{p}_1}\!\!-\!\!\mathbf{\bf{p}_2}\|\!\!+\!\!\|\mathbf{p_1}\!\!-\!\!\mathbf{p_3}\|\!\!+\!\!\|\mathbf{p_1}\!\!-\!\!\mathbf{p_4}\|
		\right) \!n^{2-\alpha}, \\
		\label{b2}
&		\sup _{x \in[0,1]}\!\!\!\left|\tilde{\psi}_n^{\prime}(x;\!\mathbf{p}_1) \tilde{\psi}_n^{\prime}(x;\! \mathbf{p}_2)\!\!-\!\!\tilde{\psi}_n^{\prime}\!\left(x ;\!\mathbf{p}_3\right) \tilde{\psi}_n^{\prime}\!\left(\!x ;\!\mathbf{p}_4\right)\!\right|\!\!\leq\!\!C\!\left(\|\mathbf{\bf{p}_1}\!\!-\!\!\mathbf{\bf{p}_2}\|\!\!+\!\!\|\mathbf{p_1}\!\!-\!\!\mathbf{p_3}\|\!\!+\!\!\|\mathbf{p_1}\!\!-\!\!\mathbf{p_4}\|
		\right) \!n^{2-\alpha},
\end{align}	
where $C$ is a positive constant depending only on $\alpha$ and $M$.
\end{lemma}

The estimates \eqref{a1} and \eqref{b1} have been already proved in \cite{SC1}, so we only need to prove \eqref{a2} and \eqref{b2}. For this purpose, it is necessary to give some auxiliary lemmas.

Using the explicit form \eqref{muna} for the eigenvalues $\mu_n(\mathbf{0})$, one can easily see that, for every $\alpha \in (2,3)$ and $r>0$, there exists $N=N(r, \alpha) \in \mathbb{N}$ such that
\begin{equation*}
\left(\mu_{n+1}(\mathbf{0})-r(n+1)^\alpha\right)-\left(\mu_n(\mathbf{0})+r n^\alpha\right)>0
\end{equation*}
for all $n \geq N$. Thus, for $r>0$ and $2<\alpha<3$, we consider the subset of the complex plane
\begin{equation*}
	C_{r, \alpha}= \left\{z \in \mathbb{C}:|z|<\mu_N(\mathbf{0})+r N^\alpha\right\}\cup\left(\bigcup_{n=N+1}^{\infty}\left\{z \in \mathbb{C}:\left|z-\mu_n(\mathbf{0})\right|<r n^\alpha\right\}\right).
\end{equation*}

The following result localizes the spectrum of $L_2(\mathbf{p})$ with respect to the set $C_{r, \alpha}$.

\begin{proposition}{\cite{Cau}}\label{th1}
	If $2<\alpha<3$ and $M>0$, then there exists $r = r(\alpha, M)>3$ such that
	$$
	\sigma(L_2(\mathbf{p})) \subset C_{r, \alpha}
	$$
	for all $\mathbf{p} \in B(\mathbf{0}, M)$.
\end{proposition}

For $\mathbf{p} \in B(\mathbf{0}, M)$ and $z \notin C_{r, \alpha}$, let $R(\mathbf{p}, z)=(L_2(\mathbf{p})-z)^{-1}$. Proposition \ref{th1} guarantees the existence of $R(\mathbf{p}, z)$.
	
Define the contour
\begin{equation}\label{gam}
	\gamma_n=\left\{z \in \mathbb{C}:\left|z-\mu_n(\mathbf{0})\right|=r n^\alpha\right\}, \quad n>N(r, \alpha).
\end{equation}
Clearly, $\mbox{int}\,\gamma_n$
contains exactly one eigenvalue $\mu_n(\mathbf{p})$ for $\mathbf{p} \in B(\mathbf{0}, M)$ if $r$ is chosen according to Proposition~\ref{th1}.

\begin{proposition}\cite{Cau}\label{le7}
	If $2<\alpha<3, M>0, \mathbf{p_1}, \mathbf{p_2} \in  B(\mathbf{0}, M)$, $r = r(\alpha, M)$, and $n>N(r, \alpha)$, then there exists a positive constant $C$, depending only on $\alpha$ and $M$, such that
	
	(1) $\sup _{z \in \gamma_n}\|R(\mathbf{p_1}, z)-R(\mathbf{p_2}, z)\|_{1, \infty} \leq C\|\mathbf{p_1}-\mathbf{p_2}\| n^{2-2 \alpha}$\\ 
and
	
	(2) $\sup _{z \in \gamma_n}\|D R(\mathbf{p_1}, z) D-D R(\mathbf{p_2}, z) D\|_{1, \infty} \leq C\|\mathbf{p_1}-\mathbf{p_2}\| n^{4-2 \alpha}$.
\end{proposition}

Relying on Proposition \ref{le7}, we prove the following lemma:
\begin{lemma}\label{le1}
If $2<\alpha<3, M>0$, $\mathbf{p_1}, \mathbf{p_2} \in B(\mathbf{0}, M)$, $r = r(\alpha, M)$, and $n>N(r, \alpha)$, then there exists a positive constant $C$, depending only on $\alpha$ and $M$, such that
\begin{align}\label{eq3}
\sup _{x, y \in[0,1]}\left|\psi_n(x; \mathbf{p_1}) \psi_n(y ; \mathbf{p_1})-\psi_n(x ; \mathbf{p_2}) \psi_n(y ; \mathbf{p_2})\right| \leq C\|\mathbf{p_1}-\mathbf{p_2}\| n^{2-\alpha},
\end{align}
\begin{align}\label{eq4}
\sup _{x, y \in[0,1]}\left|\tilde{\psi}_n^{\prime}(x; \mathbf{p_1}) \tilde{\psi}_n^{\prime}(y ; \mathbf{p_1})-\tilde{\psi}_n^{\prime}(x ; \mathbf{p_2}) \tilde{\psi}_n^{\prime}(y ; \mathbf{p_2})\right| \leq C\|\mathbf{p_1}-\mathbf{p_2}\| n^{2-\alpha}.
\end{align}
\end{lemma}
\begin{proof}

It is known that the kernel of the integral operator
	$$\frac{1}{2\pi i}\oint_{\gamma_n}R(\mathbf{p},z)dz$$
is $\psi_n(x;\mathbf{p})\psi_n(y;\mathbf{p})$ and that the kernel of the integral operator
$$\frac{1}{2\pi i}\oint_{\gamma_n}D R(\mathbf{p},z) Ddz$$ is $\psi_n'(x;\mathbf{p})\psi_n'(y;\mathbf{p})$ (see \cite{Sc} Section 6.3).
For $2<\alpha<3, M>0$, $\mathbf{p_1}, \mathbf{p_2} \in B(\mathbf{0}, M)$ and $z \in \gamma_n$, where $\gamma_n$ is defined by equation \eqref{gam}, by Proposition \ref{le7} we get
\begin{align*}
	&\sup _{x, y \in[0,1]}\left|\psi_n(x; \mathbf{p_1}) \psi_n(y ; \mathbf{p_1})-\psi_n(x ; \mathbf{p_2}) \psi_n(y; \mathbf{p_2})\right|\\&\leq
	\bigg{\|}\frac{1}{2\pi i} \oint_{\gamma_n}R(\mathbf{p_1},z)dz-\frac{1}{2\pi i} \oint_{\gamma_n}R(\mathbf{p_2},z)dz\bigg{\|}_{1,\infty}\\&\leq \frac{1}{2\pi}\oint_{\gamma_n}\|R(\mathbf{p_1},z)-R(\mathbf{p_2},z)\|_{1,\infty}dz\\&\leq C\|\mathbf{p_1}-\mathbf{p_2}\| n^{2-\alpha},
\end{align*}
where $C$ is a positive constant, which proves \eqref{eq3}.
The relation \eqref{eq4} is obtained similarly.
\end{proof}

Lemma \ref{le1} allows us to prove the following assertion:

\begin{lemma}\label{bound}
		If $2<\alpha<3, M>0$, $\mathbf{p} \in B(\mathbf{0}, M)$, $r = r(\alpha, M)$, and $n>N(r, \alpha)$, then there is a constant $Q$, depending only on $\alpha$ and $M$, such that
	\begin{equation}\label{eq14}
		\sup _{x \in[0,1]}|\psi_n(x; \mathbf{p})| \leq Q  ,
	\end{equation}
	\begin{equation}\label{eq100}
		\sup _{x \in[0,1]}|\tilde{\psi}_n^{\prime}(x; \mathbf{p})|\leq Q.
	\end{equation}
\end{lemma}
\begin{proof}
In view of \eqref{psia}, we have
$$
\psi_n(x;\mathbf{0})=\sqrt{2}\sin\bigg(\bigg(n+\frac{1}{2}\bigg)\pi x\bigg).
$$
	Letting $\mathbf{p_1}=\mathbf{p},\mathbf{p_2}=\mathbf{0}$ and $x=y$ in \eqref{eq3}, we get
	\begin{equation*}
		\bigg|\psi^2_n(x; \mathbf{p})-{2}\sin^2\bigg(\bigg(n+\frac{1}{2}\bigg)\pi x\bigg)\bigg|\leq C\|\mathbf{p}\| n^{2-\alpha}.
	\end{equation*}
	Hence 
	\begin{equation}\label{111}
		\bigg|\psi_n^2(x; \mathbf{p})\bigg|\leq 2+2CM n^{2-\alpha}\leq  2+2CM N^{2-\alpha}.
	\end{equation}
	Letting $\mathbf{p_1}=\mathbf{p}, \mathbf{p_2}=\mathbf{0}$ and $x=y$ in  \eqref{eq4}, we get
\begin{equation*}\label{eq15}
	|\tilde{[\psi]}_n^{\prime 2}(x; \mathbf{p})-{2}\cos^2\bigg(\bigg(n+\frac{1}{2}\bigg)\pi x\bigg)\bigg|\leq C\|\mathbf{p}\| n^{2-\alpha}.
\end{equation*}
Hence
\begin{equation}\label{112}
	\bigg|\tilde{[\psi]}_n^{\prime 2}(x;\mathbf{p})\bigg|\leq 2+2CM n^{2-\alpha}\leq2+2CM N^{2-\alpha}.
\end{equation}
Put $Q=(2+2CM N^{2-\alpha})^{1/2}$. Then, \eqref{111} and \eqref{112} yield the claim.
\end{proof}

\begin{lemma}\label{le4}
	If $2<\alpha<3, M>0$, $\mathbf{p_1}, \mathbf{p_2} \in B(\mathbf{0}, M)$, then there is a positive integer $P$, depending only on $\alpha$ and $M$, such that, for $n>P$, there hold
	\begin{equation}\label{eq1}
	\sup _{x \in[0,1]}\left|\psi_n(x; \mathbf{p_1})-\psi_n(x ; \mathbf{p_2})\right| \leq C\|\mathbf{p_1}-\mathbf{p_2}\| n^{2-\alpha},
	\end{equation}
	\begin{equation}\label{eq2}
	\sup _{x \in[0,1]}\left|\tilde{\psi}_n^{\prime}(x; \mathbf{p_1})-\tilde{\psi}_n^{\prime}(x ; \mathbf{p_2})\right| \leq C\|\mathbf{p_1}-\mathbf{p_2}\| n^{2-\alpha},
	\end{equation}
	where $C$ is a positive constant depending only on $\alpha$ and $M$.
\end{lemma}
\begin{proof}
	We first prove the estimate \eqref{eq1}. Let $r = r(\alpha, M)$ and $n > N(r, \alpha)$.
	Putting $\mathbf{p_2}=\mathbf{0}$ and $x=y=x_n=(4 n+2)^{-1}$ in \eqref{eq3}, we have
	$$
	\left|\psi_n^2\left(x_n; \mathbf{p_1}\right)-1\right| \leq C\|\mathbf{p_1}\|n^{2-\alpha}\leq 2CMn^{2-\alpha}.
	$$
Then, there exists $P>N(r, \alpha)$ such that, for $n\geq P$,
	$$
	\left|\psi_n^2\left(x_n; \mathbf{p_1}\right)-1\right| \leq \frac{3}{4}.
	$$
	Hence
	$$
	 \psi_n^2\left(x_n; \mathbf{p_1}\right)\geq \frac{1}{4}.
	$$
		
	To this point, there has been an ambiguity in the definition of $\psi_n(x ;\mathbf{p_1})$ by a complex factor of modulus one. We now choose $\psi_n(x ; \mathbf{p_1})$ so that $\psi_n\left(x_n ; \mathbf{p_1}\right)>0$. With this choice of $\psi_n(x ; \mathbf{p_1})$, we get
	\begin{equation}\label{eq13}
	 \psi_n\left(x_n ; \mathbf{p_1}\right)\geq \frac{1}{2} .
	\end{equation}
Setting $y=x_n$ in \eqref{eq3}, we obtain
	\begin{align}\label{eq12}
		\left|\psi_n(x ;\mathbf{p_1}) \psi_n\left(x_n ; \mathbf{p_1}\right)-\psi_n(x ; \mathbf{p_2}) \psi_n\left(x_n ; \mathbf{p_2}\right)\right|\leq C\|\mathbf{p_1}-\mathbf{p_2}\| n^{2-\alpha}.
	\end{align}
	Inequalities \eqref{eq13} and \eqref{eq12} yield that
	\begin{align}\label{eq16}
		\left|\psi_n(x ; \mathbf{p_1})-\psi_n(x ; \mathbf{p_2}) \frac{\psi_n\left(x_n ; \mathbf{p_2}\right)}{\psi_n\left(x_n ; \mathbf{p_1}\right)}\right| \leq C_1\|\mathbf{p_1}-\mathbf{p_2}\| n^{2-\alpha}.
	\end{align}
It is easy to verify that, for $a, b \geq 1 / 2$,
\begin{equation}\label{eq7}
	|a-b| \leq\left|a^2-b^2\right|.
\end{equation}
Applying inequality \eqref{eq7} to \eqref{eq3} with $x=y=x_n$ gives
\begin{equation}\label{eq11}
	\left|\psi_n\left(x_n; \mathbf{p_1}\right)-\psi_n\left(x_n ; \mathbf{p_2}\right)\right| \leq C\|\mathbf{p_1}-\mathbf{p_2}\| n^{2-\alpha}.
\end{equation}
By  \eqref{eq14}, \eqref{eq13}, \eqref{eq16} and \eqref{eq11}, we have
\begin{align*}
	&\left|\psi_n(x ; \mathbf{p_1})-\psi_n(x ; \mathbf{p_2})\right| \\&\leq\left|\psi_n(x ; \mathbf{p_1})-\psi_n(x ;\mathbf{p_2}) \frac{\psi_n\left(x_n ; \mathbf{p_2}\right)}{\psi_n\left(x_n ; \mathbf{p_1}\right)}\right|+\nonumber\left|\psi_n(x ; \mathbf{p_2})\frac{\psi_n\left(x_n ; \mathbf{p_2}\right)-\psi_n\!(x_n ; \mathbf{p_1})}{\psi_n\left(x_n ; \mathbf{p_1}\right)}\right|\\&\leq C_1\|\mathbf{p_1}-\mathbf{p_2}\| n^{2-\alpha}+2CQ\|\mathbf{p_1}-\mathbf{p_2}\| n^{2-\alpha}\\&\leq C_2\|\mathbf{p_1}-\mathbf{p_2}\| n^{2-\alpha},
\end{align*}
 which proves \eqref{eq1}.
	
	Now, consider \eqref{eq2}. By the mean value theorem, there exists $z_n \in$ $\left(0, \frac{1}{4n+2}\right)$ such that
	
	$$
	\begin{aligned}
		\tilde{\psi}_n^{\prime}\left(z_n, \mathbf{p_1}\right) & =\frac{\tilde{\psi}_n\left(x_n ; \mathbf{p_1}\right)-\tilde{\psi}_n(0 ; \mathbf{p_1})}{1 / (4n+2)}=\frac{4}{\pi} \psi_n\left(x_n ; \mathbf{p_1}\right),
	\end{aligned}
	$$
	where $\tilde{\psi}_n(0 ; \mathbf{p_1})=0$ by the  boundary conditions \eqref{boud2}. Taking \eqref{eq13} into account, we get
	$$
	\psi_n\left(x_n ; \mathbf{p_1}\right)>\frac{1}{2},
	$$
	for all $n>P$. Then
	$$
	\tilde{\psi}_n^{\prime}\left(z_n ; \mathbf{p_1}\right)>\frac{2}{\pi}>\frac{1}{2} .
	$$
Letting $\mathbf{p_2}=\mathbf{0}$ and $x=y$ in \eqref{eq4}, we obtain
\begin{equation}\label{eq5}
	\bigg|\left[\tilde{\psi}_n^{\prime}\right]^2(x ; \mathbf{p_1})-{2} \cos^2 \bigg(\bigg(n+\frac{1}{2}\bigg) \pi x\bigg)\bigg| \leq C\|\mathbf{p_1}\| n^{2-\alpha}\leq 2MC n^{2-\alpha}.
\end{equation}
Therefore, for any $x \in\bigg[0,\frac{1}{4n+2}\bigg]$,
we have
	\begin{equation}\label{eq6}
		{\left[\tilde{\psi}_n^{\prime}\right]^2(x ; \mathbf{p_1}) } \geq {2} \cos^2 \bigg(\bigg(n+\frac{1}{2}\bigg) \pi x\bigg)-2MCn^{2-\alpha} \geq 1-2MCn^{2-\alpha}\geq\frac{1}{4}
	\end{equation}
	for any $n>P$,
	where we use $
	 {2} \cos^2 \big(\big(n+\frac{1}{2}\big) \pi x\big)\geq 1$, $x \in\big[0,\frac{1}{4n+2}\big].$
	
	We know that $\tilde{\psi}_n^{\prime}\left(z_n ; \mathbf{p_1}\right)>0$ and $z_n \in(0,\frac{1}{4n+2})$. Suppose that there is a point $y_n \in[0,\frac{1}{4n+2}]$ such that $\tilde{\psi}_n^{\prime}\left(y_n ; \mathbf{p_1}\right) \leq 0$. Then, by the intermediate value
	theorem, there must exist a point $w_n \in[0,\frac{1}{4n+2}]$ such that $\tilde{\psi}_n^{\prime}\left(w_n ; \mathbf{p_1}\right)=0$. This is a contradiction with \eqref{eq6}. Thus, we conclude from \eqref{eq6} that
\begin{equation}\label{eq10}
	\tilde{\psi}_n^{\prime}(x ; \mathbf{p_1})>\frac{1}{2}, \quad \forall x \in\bigg[0,\frac{1}{4n+2}\bigg].
\end{equation}
Setting $y=x_n$ into \eqref{eq4}, we have
\begin{align}\label{eq9}
	\left|\tilde{\psi}_n^{\prime}(x ;\mathbf{p_1}) \tilde{\psi}_n^{\prime}\left(x_n ; \mathbf{p_1}\right)-\tilde{\psi}_n^{\prime}(x ; \mathbf{p_2}) \tilde{\psi}_n^{\prime}\left(x_n ; \mathbf{p_2}\right)\right| \leq C\|\mathbf{p_1}-\mathbf{p_2}\| n^{2-\alpha},
\end{align}
Inequalities \eqref{eq10} and \eqref{eq9} yield that
\begin{equation}\label{ai}
\left\lvert\, \tilde{\psi}_n^{\prime}(x ; \mathbf{p_1})-\tilde{\psi}_n^{\prime}\left(x ; \mathbf{p_2}\right) \frac{\tilde{\psi}_n^{\prime}\left(x_n ; \mathbf{p_2}\right)}{\tilde{\psi}_n^{\prime}\left(x_n ; \mathbf{p_1}\right)} \right\rvert\ \!\!\!\leq\! C_1\|\mathbf{p_1}-\mathbf{p_2}\| n^{2-\alpha}.
\end{equation}
Applying inequality \eqref{eq7} to \eqref{eq4} with $x=y=x_n$ gives
\begin{equation}\label{eq8}
	\left|\tilde{\psi}_n^{\prime}\left(x_n ; \mathbf{p_1}\right)-\tilde{\psi}_n^{\prime}\left(x_n ; \mathbf{p_2}\right)\right| \leq C\|\mathbf{p_1}-\mathbf{p_2}\| n^{2-\alpha}.
\end{equation}
By  \eqref{eq100}, \eqref{eq10}, \eqref{ai} and \eqref{eq8}, we get
\begin{align*}
	&\left|\tilde{\psi}_n^{\prime}(x ; \mathbf{p_1})-\tilde{\psi}_n^{\prime}(x ; \mathbf{p_2})\right| \\&\leq\left|\tilde{\psi}_n^{\prime}(x ; \mathbf{p_1})-\tilde{\psi}_n^{\prime}\left(x ; \mathbf{p_2}\right) \frac{\tilde{\psi}_n^{\prime}\left(x_n ; \mathbf{p_2}\right)}{\tilde{\psi}_n^{\prime}\left(x_n ; \mathbf{p_1}\right)}\right|+\nonumber\left|\tilde{\psi}_n^{\prime}(x ; \mathbf{p_2})\frac{\tilde{\psi}_n^{\prime}\left(x_n ; \mathbf{p_2}\right)-\tilde{\psi}_n^{\prime}(x_n ; \mathbf{p_1})}{\tilde{\psi}_n^{\prime}\left(x_n ; \mathbf{p_1}\right)}\right|\\&\leq C_1\|\mathbf{p_1}-\mathbf{p_2}\| n^{2-\alpha}+2CQ\|\mathbf{p_1}-\mathbf{p_2}\| n^{2-\alpha}\\&\leq C_2\|\mathbf{p_1}-\mathbf{p_2}\| n^{2-\alpha},
\end{align*}
 which proves \eqref{eq2}.
\end{proof}

\begin{proof}[Proof of Lemma \ref{prop}]
	By Lemmas \ref{bound} and \ref{le4}, it can be deduced that
\begin{align*}
			&\left| \psi_n(x; \mathbf{p_1}) \psi_n(x; \mathbf{p_2})-\psi_n\left(x ; \mathbf{p_3}\right) \psi_n\left(x ; \mathbf{p_4}\right)\right|\\&\leq|\psi_n(x; \mathbf{p_1})(\psi_n(x; \mathbf{p_2})-\psi_n(x ; \mathbf{p_1}))|
			+|\psi_n(x; \mathbf{p_1}) (\psi_n(x; \mathbf{p_1})-\psi_n(x ; \mathbf{p_3}))|+\\&\quad\;|\psi_n(x; \mathbf{p_3})(\psi_n(x; \mathbf{p_1})-\psi_n(x ;\mathbf{p_4}))|
			\\&\leq C\left(\|\mathbf{p_1}-\mathbf{p_2}\|+\|\mathbf{p_1}-\mathbf{p_3}\|+\|\mathbf{p_1}-\mathbf{p_4}\|
			\right) \!n^{2-\alpha},
\end{align*}
Analogously, we have
\begin{align*}
		&\left| \tilde{\psi}'_n(x; \mathbf{p_1} \tilde{\psi}'_n(x; \mathbf{p_2})-\tilde{\psi}'_n\left(x ; \mathbf{p_3}\right) \tilde{\psi}'_n\left(x ; \mathbf{p_4}\right)\right|\\&\leq |\tilde{\psi}'_n\!(x; \mathbf{p_1})(\tilde{\psi}'_n(x; \mathbf{p_2})-\tilde{\psi}'_n(x ;\mathbf{p_1}))|
		+|\tilde{\psi}'_n(x; \mathbf{p_1})(\tilde{\psi}'_n(x; \mathbf{p_1})-\tilde{\psi}'_n(x;\mathbf{p_3}))|+\\&\quad\;|\tilde{\psi}'_n(x; \mathbf{p_3})(\tilde{\psi}'_n(x; \mathbf{p_1})-\tilde{\psi}'_n(x;\mathbf{p_4}))|
		\\&\leq C\!\left(\|\mathbf{p_1}-\mathbf{p_2}\|+\|\mathbf{p_1}-\mathbf{p_3}\|+\|\mathbf{p_1}-\mathbf{p_4}\|
		\right) \!n^{2-\alpha},
\end{align*}
which proves the conclusions \eqref{a2} and \eqref{b2}.
\end{proof}

\section{Proof of Theorem \ref{thm}} \label{sec:proof}

The proof of Theorem \ref{thm} hinges on the following key observation.

\begin{lemma}\label{riesz}
	Let  $\mathbf{a}$ be a constant and $\mathbf{a}\in \mathbf{W}$, then there exists $\varepsilon>0$ such that, if $\mathbf{p_1}$, $\mathbf{p_2}\in B(\mathbf{a},\varepsilon)$, the systems
 $$\big\{1\big\}\cup\bigg\{\sqrt{2}(1-\phi_n(x;\mathbf{p_1}){\phi}_n(x;\mathbf{p_2})),\sqrt{2}(1-\psi_n(x;\mathbf{p_1}){\psi}_n(x;\mathbf{p_2}))\bigg\}_{n\geq 1}$$ and $$\big\{1\big\}\cup\bigg\{2\sqrt{2}\bigg(\frac{\phi_n'(x;\mathbf{p_1})\phi_n'(x;\mathbf{p_2})}{2n^2\pi^2}-\frac{1}{2}\bigg),2\sqrt{2}\bigg(\frac{\psi_n'(x;\mathbf{p_1})\psi'_n(x;\mathbf{p_2})}{2(n+\frac{1}{2})^2\pi^2}-\frac{1}{2}\bigg)\bigg\}_{n\geq 1}$$ are Riesz bases in $L^2(0,1)$.
\end{lemma}
\begin{proof}
 For $n\geq 1$, denote
 \begin{equation*}
 	e_0=1,
 \end{equation*}
 \begin{equation*}\label{6_1}
 	e_{2n}=2\sqrt{2}\bigg(\frac{1}{2n^2\pi^2}\phi_n'(x;\mathbf{a})\phi_n'(x;\mathbf{a})-\frac{1}{2}\bigg)=\sqrt{2}\cos (2n\pi x),
 \end{equation*}
 \begin{equation*}\label{6_2}
	e_{2n-1}=2\sqrt{2}\bigg(\frac{1}{2\big(n+\frac{1}{2}\big)^2\pi^2}\psi_n'(x;\mathbf{a})\psi_n'(x;\mathbf{a})-\frac{1}{2}\bigg)=\sqrt{2}\cos ((2n+1)\pi x),
\end{equation*}
and
\begin{equation*}
	d_0=1,
\end{equation*}
\begin{equation*}\label{6_3}
	d_{2n}=2\sqrt{2}\bigg(\frac{1}{2n^2\pi^2}\phi_n'(x;\mathbf{p_1})\phi_n'(x;\mathbf{p_2})-\frac{1}{2}\bigg),
\end{equation*}
\begin{equation*}\label{6_4}
	d_{2n-1}=2\sqrt{2}\bigg(\frac{1}{2\big(n+\frac{1}{2}\big)^2\pi^2}\psi_n'(x;\mathbf{p_1})\psi_n'(x;\mathbf{p_2})-\frac{1}{2}\bigg).
\end{equation*}

Note that $\{e_{n}\}_{n\geq 0}=\{1\}\cup\big\{\sqrt{2}\cos (n\pi x)\}_{n\geq 1}$ is an orthonormal basis in $L^2(0,1)$. Let $M=\|\mathbf{a}\|+1$. Choose  $$\alpha=2.9,\;\;\;\; \varepsilon_{2P+1}\!\!=\!\!\frac{1}{32\sqrt{\sum\limits_{n=2P+1}^{\infty}n^{-1.8}}},$$ where $P$ is from Lemma \ref{prop}. If $\mathbf{p_1}$, $\mathbf{p_2}\in B(\mathbf{a},\varepsilon_{2P+1})$, then, obviously, $\mathbf{p_1}$, $\mathbf{p_2}\in B(\mathbf{0},M)$. By Lemma \ref{prop}, we have
\begin{align*}
	&\sum_{n=2P+1}^{\infty} ||d_n-e_n||^2\\&\leq8\sum_{n=2P+1}^{\infty}(\|\mathbf{p_1}-\mathbf{p_2}\|+2\|\mathbf{p_1}-\mathbf{a}\|)^2 n^{-1.8}\\&\leq8\sum_{n=2P+1}^{\infty}(\|\mathbf{p_2}-\mathbf{a}\|+3\|\mathbf{p_1}-\mathbf{a}\|)^2 n^{-1.8}\\&\leq8\cdot(4\varepsilon_{2P+1})^2\sum_{n=2P+1}^{\infty}n^{-1.8}\\&<\frac{1}{2}.
\end{align*}

Put
$$
\delta=2\sqrt{1+\frac{1}{32\sqrt{P}}}-2.
$$
By Proposition \ref{le2}, $\phi_n'(x;\mathbf{p})$ and $\psi_n'(x;\mathbf{p})$ are both analytically dependent on $\mathbf{p}$ for $n=\overline{1,P}$. Consequently, there exists $\{\varepsilon_{n}\}_{n=\overline{1,2P}}$ such that,
if $\mathbf{p}\in B(\mathbf{a},\varepsilon_{2n})$, $n=\overline{1,P}$, then 
\begin{equation}\label{11}
	\|\phi_n'(x;\mathbf{p})-\phi_n'(x;\mathbf{a})\|\leq \delta,
\end{equation}
so
\begin{equation}\label{12}
	\|\phi_n'(x;\mathbf{p})\|<\delta+\|\phi_n'(x;\mathbf{a})\|<\delta+n\pi;
\end{equation} 
and, if $\mathbf{p}\in B(\mathbf{a},\varepsilon_{2n-1})$, $n=\overline{1,P}$, then
\begin{equation}\label{13}
	\|\psi_n'(x;\mathbf{p})-\psi_n'(x;\mathbf{a})\|\leq \delta,
\end{equation}
so
\begin{equation}\label{14}
	\|\psi_n'(x;\mathbf{p})\|<\delta+\|\psi_n'(x;\mathbf{a})\|<\delta+(n+\frac{1}{2}\pi).
\end{equation}

Choosing $\varepsilon_0=\min\{\varepsilon_1, \varepsilon_2,\dots, \varepsilon_{2P}\}$, letting $\mathbf{p_1}$, $\mathbf{p_2}\in B(\mathbf{a}, \varepsilon_0)$, from \eqref{11} and \eqref{12}, we have
\begin{align*}
	&\|d_{2n}-e_{2n}\|\\&=\frac{8}{2n^2\pi^2}\|\phi_n'(x;\mathbf{p_1})\phi_n'(x;\mathbf{p_2})-\phi_n'(x;\mathbf{a})\phi_n'(x;\mathbf{a})\|\\
	&\leq \frac{8}{2n^2\pi^2}\!\big(\|\phi_n'\!(\!x;\mathbf{p_1})\phi_n'\!(\!x;\mathbf{p_2})\!\!-\!\!\phi_n'\!(\!x;\mathbf{a})\phi_n'\!(\!x;\mathbf{p_2})\!\|\!\!+\!\!\|\phi_n'\!(\!x;\mathbf{a})\phi_n'\!(\!x;\mathbf{p_2})\!\!-\!\!\phi_n'\!(\!x;\mathbf{a})\phi_n'\!(x;\mathbf{a})\!\|\big)\\
	&\leq\frac{8}{2n^2\pi^2}(\|\phi_n'(\!x;\mathbf{p_2})\|\|\phi_n'(x;\mathbf{p_1})-\phi_n'(x;\mathbf{a})\|+\|\phi_n'(x;\mathbf{a})\|\|\phi_n'(x;\mathbf{p_2})-\phi_n'(x;\mathbf{a})\|)\\
	&\leq\frac{4(\delta+n\pi)}{n^2\pi^2}\big(\|\phi_n'(x;\mathbf{p_1})-\phi_n'(x;\mathbf{a})\|+\|\phi_n'(x;\mathbf{p_2})-\phi_n'(x;\mathbf{a})\|\big)\\
	&<\frac{8\delta(\delta+n\pi)}{n^2\pi^2}\\
	&<\frac{8\delta^2+32\delta}{9}\\
	&=\frac{8}{9}\bigg(8+\frac{4}{32\sqrt{P}}-8\sqrt{1+\frac{1}{32\sqrt{P}}}+8\sqrt{1+\frac{1}{32\sqrt{P}}}-8\bigg)\\&=\frac{1}{9\sqrt{P}}.
\end{align*}
Analogously, from \eqref{13} and \eqref{14}, we obtain
\begin{align*}
	&\|d_{2n-1}-e_{2n-1}\|\\&=\!\frac{8}{2\big(n+\frac{1}{2}\big)^2\pi^2}\|\psi_n'(x;\mathbf{p_1})\psi_n'(x;\mathbf{p_2})-\psi_n'(x;\mathbf{a})\psi_n'(x;\mathbf{a})\|\\
	&\leq \frac{8}{2\big(n\!\!+\!\!\frac{1}{2}\big)^2\!\pi^2}\!\big(\|\!\psi_n'\!(\!x;\mathbf{p_1}\!)\psi_n'\!(\!x;\mathbf{p_2}\!)\!\!-\!\!\psi_n'\!(\!x;\mathbf{a}\!)\psi_n'\!(\!x;\mathbf{p_2}\!)\!\|\!\!+\!\!\|\!\psi_n'\!(\!x;\mathbf{a}\!)\psi_n'\!(\!x;\mathbf{p_2}\!)\!\!-\!\!\psi_n'\!(\!x;\mathbf{a}\!)\psi_n'\!(\!x;\mathbf{a})\!\|\big)\\
	&\leq \frac{8}{2\big(n\!+\!\frac{1}{2}\big)^2\pi^2}\big(\|\psi_n'(x;\mathbf{p_2})\|\|\psi_n'(x;\mathbf{p_1})\!\!-\!\!\psi_n'(x;\mathbf{a})\|\!\!+\!\!\|\psi_n'(x;\mathbf{a})\|\|\psi_n'(x;\mathbf{p_2})\!\!-\!\!\psi_n'(x;\mathbf{a})\|\big)\\
	&\leq\frac{4\big(\delta+\big(n+\frac{1}{2}\big)\pi\big)}{\big(n+\frac{1}{2}\big)^2\pi^2}\big(\|\psi_n'(x;\mathbf{p_1})-\psi_n'(x;\mathbf{a})|+|\psi_n'(x;\mathbf{p_2})-\psi_n'(x;\mathbf{a})\|\big)\\
	&<\frac{8\delta\big(\delta+\big(n+\frac{1}{2}\big)\pi\big)}{\big(n+\frac{1}{2}\big)^2\pi^2}\\
	&<\frac{8\delta^2+32\delta}{9}\\
	&=\frac{1}{9\sqrt{P}}.
\end{align*}

Thus, if $\mathbf{p_1}$, $\mathbf{p_2}\in B(\mathbf{a}, \varepsilon_0)$, we have
\begin{align*}
	\sum_{n=0}^{2P} \| d_n-e_n \|^2<\sum_{n=1}^{2P} \frac{1}{{81P}}=\frac{2}{81}<\frac{1}{2}.
\end{align*}

Set $\varepsilon=\min\{\varepsilon_0, \varepsilon_{2P+1}\}$. For $\mathbf{p_1}$, $\mathbf{p_2}\in B(\mathbf{a},\varepsilon)$, the above calculation implies that
\begin{align*}
	\sum_{n=0}^{\infty} \| d_n-e_n \|^2<1.
\end{align*}
By Proposition \ref{le3}, the sequence $\{d_{n}\}_{n\geq 0}$ is a Riesz basis in $L^2(0,1)$. Consequently, $$\big\{1\big\}\cup\bigg\{2\sqrt{2}\bigg(\frac{\phi_n'(x;\mathbf{p_1})\phi_n'(x;\mathbf{p_2})}{2n^2\pi^2}-\frac{1}{2}\bigg),2\sqrt{2}\bigg(\frac{\psi_n'(x;\mathbf{p_1})\psi'_n(x;\mathbf{p_2})}{2(n+\frac{1}{2})^2\pi^2}-\frac{1}{2}\bigg)\bigg\}_{n\geq 1}$$
is a Riesz basis in $L^2(0,1)$.

The fact that the sequence $$\big\{1\big\}\cup\bigg\{\sqrt{2}(1-\phi_n(x;\mathbf{p_1}){\phi}_n(x;\mathbf{p_2})),\sqrt{2}(1-\psi_n(x;\mathbf{p_1}){\psi}_n(x;\mathbf{p_2}))\bigg\}_{n\geq 1}$$ also forms a Riesz basis in $L^2(0,1)$ can be established similarly.
\end{proof}

\begin{proof}[Proof of the Theorem \ref{thm}]
Since
\begin{equation} \label{eqlan}
	\lambda_n(\mathbf{p}_1)={\lambda}_n(\mathbf{p}_2) \text { for all } n,
\end{equation}
using the sesquilinear form \eqref{frakq}, we obtain
\begin{align*}
		&\mathfrak{q}_{(p_1,q_1)}[{\phi}_{n}(x;\mathbf{p_1}),{\phi}_{n}(x;\mathbf{p_2})]=({\phi}_{n}''(x;\mathbf{p_1}),{\phi}_{n}''(x;\mathbf{p_2}))+(p_1{\phi}_{n}'(x;\mathbf{p_1}),{\phi}_{n}'(x;\mathbf{p_2}))\\&+(q_1{\phi}_{n}(x;\mathbf{p_1}),{\phi}_{n}(x;\mathbf{p_2}))=\lambda_n(\mathbf{p}_1)({\phi}_{n}(x;\mathbf{p_1}),{\phi}_{n}(x;\mathbf{p_2}))
\end{align*}
and
\begin{align*}
	&\mathfrak{q}_{(p_2,q_2)}[{\phi}_{n}(x;\mathbf{p_1}),{\phi}_{n}(x;\mathbf{p_2})]=({\phi}_{n}''(x;\mathbf{p_1}),{\phi}_{n}''(x;\mathbf{p_2}))+(p_2{\phi}_{n}'(x;\mathbf{p_1}),{\phi}_{n}'(x;\mathbf{p_2}))\\&+(q_2{\phi}_{n}(x;\mathbf{p_1}),{\phi}_{n}(x;\mathbf{p_2}))=\lambda_n(\mathbf{p}_2)({\phi}_{n}(x;\mathbf{p_1}),{\phi}_{n}(x;\mathbf{p_2})).
\end{align*}
Subtracting the second relation from the first one, we get
\begin{align}\label{43}
&(p_1{\phi}_{n}'(x;\mathbf{p_1}),{\phi}_{n}'(x;\mathbf{p_2}))+(q_1{\phi}_{n}(x;\mathbf{p_1}),{\phi}_{n}(x;\mathbf{p_2}))=(p_2{\phi}_{n}'(x;\mathbf{p_1}),{\phi}_{n}'(x;\mathbf{p_2}))\nonumber\\&+(q_2{\phi}_{n}(x;\mathbf{p_1}),{\phi}_{n}(x;\mathbf{p_2})).
\end{align}
Analogously, since
	\begin{equation*} 
	\mu_n(\mathbf{p}_1)={\mu}_n(\mathbf{p}_2) \text { for all } n,
	\end{equation*}
	we have
	\begin{align}\label{44}
	&(p_1{\psi}_{n}'(x;\mathbf{p_1}),{\psi}_{n}'(x;\mathbf{p_2}))+(q_1{\psi}_{n}(x;\mathbf{p_1}),{\psi}_{n}(x;\mathbf{p_2}))=(p_2{\psi}_{n}'(x;\mathbf{p_1}),{\psi}_{n}'(x;\mathbf{p_2}))\nonumber\\&+(q_2{\psi}_{n}(x;\mathbf{p_1}),{\psi}_{n}(x;\mathbf{p_2})).
	\end{align}

	\textbf{Part (1)}	
	Since $q_1(x)=q_2(x)$, using \eqref{43} and \eqref{44}, we obtain
	\begin{equation}\label{5}
		\int_0^1({p_1(x)-p_2(x)})\phi_{n}'(x;\mathbf{p_1})\phi_{n}'(x;\mathbf{p_2})dx=0.
	\end{equation}
	and
	\begin{equation}\label{6}
	\int_0^1({p_1(x)-p_2(x)})\psi_{n}'(x;\mathbf{p_1})\psi_{n}'(x;\mathbf{p_2})dx=0,
\end{equation}
respectively.

 The equalities \eqref{eqlan} and the eigevalue asymptotics from  	
 \cite{Sc}:
 \begin{equation*}
 	\lambda_n(\mathbf{p}_i)=n^4\pi^4+n^2\pi^2\bigg(\int_{0}^{1}p_i(x)dx\bigg)+o(n^{\frac{3}{2}}), \quad i = 1, 2,
 \end{equation*}
 together imply 
 \begin{equation}\label{7}
 	\int_{0}^{1}(p_1(x)-p_2(x))dx=0.
 \end{equation}
From $\eqref{5}$ and $\eqref{7}$ we get
\begin{equation*}
	2\sqrt{2}\int_0^1({p_1(x)-p_2(x)})\bigg(\frac{\phi_n'(x;\mathbf{p_1})\phi_n'(x;\mathbf{p_2})}{2n^2\pi^2}-\frac{1}{2}\bigg)dx=0.
\end{equation*}
Similarly, from $\eqref{6}$ and $\eqref{7}$, we obtain
\begin{equation*}
	2\sqrt{2}\int_0^1({p_1(x)-p_2(x)})\bigg(\frac{\psi_n'(x;\mathbf{p_1})\psi'_n(x;\mathbf{p_2})}{2(n+\frac{1}{2})^2\pi^2}-\frac{1}{2}\bigg)dx=0.
\end{equation*}
 By Lemma \ref{riesz}, the collection $$\big\{1\big\}\cup\bigg\{2\sqrt{2}\bigg(\frac{\phi_n'(x;\mathbf{p_1})\phi_n'(x;\mathbf{p_2})}{2n^2\pi^2}-\frac{1}{2}\bigg),2\sqrt{2}\bigg(\frac{\psi_n'(x;\mathbf{p_1})\psi'_n(x;\mathbf{p_2})}{2(n+\frac{1}{2})^2\pi^2}-\frac{1}{2}\bigg)\bigg\}_{n\geq 1}$$ is a Riesz basis in $L^2(0,1)$, which implies that $p_1(x)=p_2(x)$ a.e. on $(0,1)$ and so proves part (1) of Theorem~\ref{thm}.
 
 \textbf{Part (2)}	
By \eqref{43} we obtain
 \begin{equation}\label{51}
 	\int_0^1({q_1(x)-q_2(x)})\phi_{n}(x;\mathbf{p_1})\phi_{n}(x;\mathbf{p_2})dx=0,
 \end{equation}
since $p_1(x)=p_2(x)$.
From \eqref{a1} we have
\begin{equation}\label{3.4}
	\sup _{x \in[0,1]}\left| \phi_{n}(x;\mathbf{p_1})\phi_{n}(x;\mathbf{p_2})-\phi_n^{2}(x; \mathbf{0})\right| 
	\leq
	C(\|\mathbf{p_1}-\mathbf{p_2}\|+2\|\mathbf{p_1}\|) n^{2-\alpha},
\end{equation}
where $2<\alpha<3$.
Substituting \eqref{3.4} into \eqref{51}, we obtain
\begin{align*}
	0&=\int_{0}^{1}(q_1(x)-q_2(x))\phi_{n}(x;\mathbf{p_1})\phi_{n}(x;\mathbf{p_2})dx\\&=\int_{0}^{1}(q_1(x)-q_2(x))[O(n^{2-\alpha})+2\sin^2n\pi x]dx\\&\nonumber=O(n^{2-\alpha})\!\!\int_{0}^{1}\!\!\!(q_1(x)\!-\!q_2(x))dx\!+\!\!\!\int_{0}^{1}\!(q_1(x)\!-\!\!q_2(x))dx\!-\!\!\!\int_{0}^{1}\!\!(q_1(x)\!-\!q_2(x))\cos (2n\pi x)dx.
\end{align*}
Let $n\rightarrow\infty$ in the above relation. Since $q_1(x)$, $q_2(x)\in L^{2}(0,1)$ and $2<\alpha<3$, according to Riemann-Lebesgue's theorem, we get
\begin{equation}\label{int}
	\int_{0}^{1}(q_1(x)-q_2(x))dx=0.
\end{equation}
 Similarly, by \eqref{44} we obtain
  \begin{equation}\label{52}
 	\int_0^1({q_1(x)-q_2(x)})\psi_{n}(x;\mathbf{p_1})\psi_{n}(x;\mathbf{p_2})dx=0.
 \end{equation}
From $\eqref{51}$ and $\eqref{int}$ we get
\begin{equation}\label{9}
		\sqrt{2}\int_0^1({q_1(x)-q_2(x)})(1-\phi_n(x;\mathbf{p_1}){\phi}_n(x;\mathbf{p_2}))dx=0.
\end{equation}
Similarly, the relations $\eqref{int}$ and $\eqref{52}$ imply
\begin{equation}\label{99}
	\sqrt{2}\int_0^1({q_1(x)-q_2(x)})(1-\psi_n(x;\mathbf{p_1}){\psi}_n(x;\mathbf{p_2}))dx=0.
\end{equation}
 By Lemma \ref{riesz}, $$\{1\}\cup\{\sqrt{2}(1-\phi_n(x;\mathbf{p_1}){\phi}_n(x;\mathbf{p_2})),\sqrt{2}(1-\psi_n(x;\mathbf{p_1}){\psi}_n(x;\mathbf{p_2}))\}_{n\geq 1}$$ is a Riesz basis in $L^2(0,1)$. Therefore, the relations \eqref{int}, \eqref{9}, and \eqref{99} imply that $q_1(x)=q_2(x)$ a.e. on $(0,1)$. 
\end{proof}

\vspace{1mm}
\th{Conflict of Interest} {\rm The authors declare no conflict of interest.}
\par\vspace{1mm}

\th{Acknowledgements} {\rm Ai-wei Guan and Chuan-fu Yang were supported by the National Natural Science Foundation of China (Grant No. 11871031) and by the Natural Science Foundation of the Jiangsu Province of China (Grant No. BK 20201303).} %The authors thank the referees for their insightful comments and helpful suggestions. }

\end{document}